\documentclass[preprint,12pt]{elsarticle}

\usepackage{amssymb}
\usepackage{amsmath}
\usepackage{graphicx} 
\usepackage{epstopdf} 
\usepackage{placeins}

\newtheorem{theorem}{Theorem}[section]
\newtheorem{remark}[theorem]{Remark}
\newtheorem{lemma}[theorem]{Lemma}

\newtheorem{proposition}[theorem]{Proposition}

\journal{Applied and Computational Harmonic Analysis}

\begin{document}

\begin{frontmatter}



\title{Phase Retrieval of Bandlimited Signals Using a Morlet Wavelet and an Admissible Companion}
\tnotetext[t1]{This work was supported partially by the National Natural Science Foundation of China (12471131 and 11971348) and Tianjin Natural Science Foundation Project (24JCYBJC00480).}
\author[1]{Bei Liu}
\ead{liubei1@mail.nankai.edu.cn}
\author[1]{Meng Ning}
\ead{ningmeng1341@stud.tjut.edu.cn}
\author[1]{Liping Wu\corref{cor1}}
\ead{wuliping@email.tjut.edu.cn}
\author[1]{Qingyue Zhang}
\ead{jczhangqingyue@163.com}
\cortext[cor1]{Corresponding author}

\affiliation[1]{organization={School of Mathematical Sciences, Institute of Signal Processing and Machine Learning, Tianjin University of Technology},
	city={Tianjin},
	postcode={300384}, 
	country={China}}

\begin{abstract}
In this paper, we demonstrate that complex bandlimited signals can be uniquely recovered (up to a global phase factor) from two families of wavelet transform magnitudes, sampled at twice the Nyquist rate in two frequency bins. These families come from translations of the fixed dilated Morlet wavelet and a fixed companion wavelet. The companion wavelet admits an augmented Morlet--Gaussian representation. We rewrite the two families as two modulated Gaussian convolution measurements of an auxiliary bandlimited signal. Using a Gaussian convolution uniqueness result, we prove that samples on \((4B)^{-1}\mathbb{Z}\) determine every signal in \(PW_B^2\) (up to a global phase factor). Compared with earlier wavelet phase retrieval results, our approach requires significantly fewer sampling points.
\end{abstract}

%


\end{frontmatter}






\section{Introduction}

The wavelet transform was introduced in the work of Morlet and coauthors in 1982 \cite{ref2,ref20}.
Through translations and dilation, it provides a time-frequency description of signals. As a highly valuable mathematical tool, it has also been widely used in signal analysis and image processing \cite{ref4,ref5,ref6}. 

Phase retrieval is the problem of recovering a signal (up to a constant global phase factor) from magnitudes of linear measurements, also known as intensity measurements or magnitude-only measurements. The problem originated in crystallography \cite{sayre} and now arises in adaptive optics \cite{B3}, optical wavefront reconstruction \cite{B4}, signal processing \cite{J1}, radar \cite{J2}, and speech recognition \cite{J3}.
The measurements may be Fourier, Gabor, or wavelet-transform measurements; see, for example, \cite{WA1,WA2,WA3,WA4,gaborconnection,ref22,ref23,ref24,ref25}.

For wavelet phase retrieval, in 2015, Mallat and Waldspurger proved a uniqueness result for Cauchy wavelets \cite{ref10}. In 2023, Alaifari, Bartolucci, and Wellershoff studied phase retrieval of bandlimited functions for wavelets with finitely many vanishing moments \cite{ref1}. Furthermore, in 2024, they also proved that every real-valued signal can be uniquely recovered (up to global sign) from the magnitudes of its wavelet frame coefficients when the analyzing wavelet is real-valued \cite{ref12}.
Related sampled Gabor phase-retrieval results include \cite{ref15}.

In our paper, we focus on the phase retrieval from the wavelet transforms measurements when the recovered signals are the bandlimited square-integrable signals $f\in L^{2}(\mathbb{R}).$ The wavelet transforms of $f\in L^{2}(\mathbb{R})$ is defined by \cite{ref1}
\begin{eqnarray}\label{eq:wavelet-definition}
	\mathcal{W}_{\psi}f(b,a):=a^{-\frac{1}{2}}\int_{\mathbb{R}}f(x)\overline{\psi\left(\frac{x-b}{a}\right)}dx,b\in\mathbb{R},a\in\mathbb{R}_{+}.
\end{eqnarray}
In this context, a non-zero function $\psi\in L^{1}(\mathbb{R})$ is called a wavelet if $\hat{\psi}(0)=0$, or equivalently if
$$\int_{\mathbb{R}}\psi(x)dx=0.$$

We will recover the signal $f$ from the magnitude-only measurements $|\mathcal{W}_{\psi}f(b,a)|$, known as wavelet phase retrieval. Our main result establishes that the bandlimited square-integrable signals, with a bandwidth $B>0$, can be recovered (up to a constant global phase factor) from coefficient magnitudes associated with translations of the fixed Morlet wavelet \(D_{a}\psi\) and an explicitly constructed fixed companion wavelet \(h_{a,w}\) in (\ref{eq:companion-window}). It is well known that the Morlet wavelet is
\begin{eqnarray}\label{eq:intro-morlet}
	\psi(x)=\pi^{-\frac{1}{4}}[e^{it_{0}x}-e^{-t_{0}^{2}/2}]e^{-x^{2}/2},\quad t_{0}\in\mathbb{R}\setminus\{0\},x\in\mathbb{R}.
\end{eqnarray}
Indeed, its Fourier transform can be factored near the origin as
\begin{equation}\label{eq:morlet-fourier}
	\widehat{\psi}(\xi)=\sqrt{2}\,\pi^{1/4}e^{-t_{0}^{2}/2} e^{-2\pi^{2}\xi^{2}}\bigl(e^{2\pi t_{0}\xi}-1\bigr).
\end{equation}
Obviously, $\widehat{\psi}$ is real valued and \(\widehat{\psi}(0)=0\).

The remainder of this paper is organized as follows.   Section 2 collects the necessary notations, definitions, and auxiliary lemmas. Section 3 establishes a uniqueness result for bandlimited signals from the magnitudes of their convolutions with a Gaussian kernel. In Section 4, we state the main result to show that every signal can be uniquely recovered (up to global sign) from the magnitudes associated with the fixed Morlet wavelet and an explicitly constructed fixed companion wavelet. Finally, in Section 5, we conclude the paper.

\section{Preliminaries}

\subsection{Basic notation and definitions}
Let us give some basic notations and definitions used in this paper.

$\bullet$ The Hilbert-space inner product is
\[
\langle f,g\rangle=\int_{\mathbb{R}}f(x)\overline{g(x)}\,dx.
\]

$\bullet$ For \(y,\omega\in\mathbb{R}\) and \(s>0\), we define translation operator $T_{y}$, the modulation operator $M_{w}$ and the dilation operator $D_{s}$ by
\[
(T_{y}f)(x)=f(x-y),\qquad (M_{\omega} f)(x)=e^{2\pi i\omega x}f(x),\qquad (D_{s}f)(x)=s^{-1/2}f(x/s).
\]

$\bullet$ For $f\in L^{1}(\mathbb{R})\cap L^{2}(\mathbb{R})$, the Fourier transform of $f$ is defined by
$$\mathcal{F}f(t)=\widehat{f}(t)=\int_{\mathbb{R}}f(x)e^{-2\pi ixt}dx,\;t\in\mathbb{R}.$$
The Fourier transform is extended to \(L^{2}(\mathbb{R})\) by Plancherel's theorem.

$\bullet$ For $f\in L^{1}(\mathbb{R})\cap L^{2}(\mathbb{R})$, the Fourier inverse transform of $f$ is
$$\mathcal{F}^{-1}f(x)=\check{f}(x)=\int_{\mathbb{R}}f(t)e^{2\pi ixt}dt,\;x\in\mathbb{R}.$$
With these conventions,
\begin{equation}\label{eq:operator-fourier}
	\widehat{T_{y}f}(\xi)=e^{-2\pi iy\xi}\widehat{f}(\xi),\qquad \widehat{M_{\omega} f}(\xi)=\widehat{f}(\xi-\omega),\qquad \widehat{D_{s}f}(\xi)=s^{1/2}\widehat{f}(s\xi).
\end{equation}
The last identity follows directly from the change of variables \(x=su\):
\[
\widehat{D_{s}f}(\xi)=s^{-1/2}\int_{\mathbb{R}}f(x/s)e^{-2\pi ix\xi}\,dx=s^{1/2}\int_{\mathbb{R}}f(u)e^{-2\pi isu\xi}\,du.
\]

$\bullet$ Define the convolution $f\ast g$ by $$(f\ast g)(x)=\int_{\mathbb{R}}f(t)g(x-t)dt.$$
This integral is absolutely convergent for every \(x\), since
\[
\int_{\mathbb{R}}|f(t)g(x-t)|\,dt\leq\|f\|_{2}\|g\|_{2}.
\]
These hypotheses apply to every convolution below.

$\bullet$ The $Paley-Wiener$ space of bandlimited functions is defined by
$$PW^{p}_{B}:=\{f\in L^{p}(\mathbb{R}):\mathrm{supp}(\mathcal{F}f)\subset[-B,B]\},$$ where $p\in[1,\infty],\:B>0$.
Equation \eqref{eq:operator-fourier} gives the bandwidth rule
\begin{equation}\label{eq:bandwidth-rule}
	f\in PW_{B}^{2}\quad\Longrightarrow\quad D_{s}f\in PW_{B/s}^{2} \quad\text{and}\quad D_{1/s}f\in PW_{sB}^{2}.
\end{equation}

$\bullet$ The equivalence relation $\sim$ on the set of complex-valued functions is defined by
\begin{eqnarray*}
	f\sim g\Leftrightarrow \exists\:\alpha\in\mathbb{R}:f=e^{i\alpha}g,
\end{eqnarray*}
we say that $f$ and $g$ agree up to global phase if $f\sim g$. 

$\bullet$ The Guassian function is given by
\begin{equation*}\label{eq:gaussian}
	\phi(x)=2^{1/4}e^{-\pi x^{2}}.
\end{equation*}
Then \(\|\phi\|_{2}=1\), \(\widehat{\phi}=\phi\), and \(\phi\) is real and even.

$\bullet$ The Schwartz space is defined by
\[
\mathcal{S}(\mathbb{R})
=
\left\{
f\in C^\infty(\mathbb{R}):
\sup_{x\in\mathbb{R}}|x^m f^{(n)}(x)|<\infty
\ \text{for all } m,n\in\mathbb{N}_0
\right\},
\]
where \(\mathbb{N}_0=\{0,1,2,\dots\}\).

\subsection{Auxiliary results}

The following result is the uniqueness input used in the paper.
\begin{proposition}[Gaussian two-bin uniqueness]\label{prop:gaussian-uniqueness}
	\cite[Theorem 6]{ref15}
	Let \(B>0\), and let \(w_{0},w_{1}\in\mathbb{R}\) with \(w_{0}\neq w_{1}\).
	If \(f,g\in PW_{B}^{2}\) satisfy
	\[
	\bigl|\phi*(M_{-w_{j}}f)\bigr|(x)=\bigl|\phi*(M_{-w_{j}}g)\bigr|(x), \qquad j=0,1,\quad x\in\frac{1}{4B}\mathbb{Z},
	\]
	then \(f\sim g\).
\end{proposition}

The ordering assumption \(\omega_{0}\neq\omega_{1}\) in Proposition \ref{prop:gaussian-uniqueness} is called two distinct frequency bins in \cite{ref15}.
\begin{proposition}[Plancherel]\label{prop:plancherel}
	\cite[Theorem~1.1.2]{ref14}
	The Fourier transform extends to a unitary operator on \(L^{2}(\mathbb{R})\).
	In particular,
	\[
	\|f\|_{2}=\|\widehat{f}\|_{2}, \qquad \langle f,g\rangle=\langle\widehat{f},\widehat{g}\rangle \quad(f,g\in L^{2}(\mathbb{R})).
	\]
\end{proposition}

The following lemma states the commutativity properties of the modulation operators with respect to convolution.
\begin{lemma}\label{lem:modulation-convolution}
	For \(f,g\in L^{2}(\mathbb{R})\) and \(\omega\in\mathbb{R}\),
	\[
	(f*M_{\omega} g)(x)=M_{\omega}\bigl((M_{-\omega}f)*g\bigr)(x).
	\]
\end{lemma}
\begin{proof}
	For every \(x\in\mathbb{R}\), absolute convergence permits the calculation
	\begin{align*}
		(f*M_{\omega} g)(x) &=\int_{\mathbb{R}}f(t)e^{2\pi i\omega(x-t)}g(x-t)\,dt\\
		&=e^{2\pi i\omega x}\int_{\mathbb{R}}(M_{-\omega}f)(t)g(x-t)\,dt.
	\end{align*}
	This is the asserted identity.
\end{proof}

Next, we present the commutativity properties of the dilation operators with
respect to convolution.
\begin{lemma}\label{lem:dilation-convolution}
	For \(f,g\in L^{2}(\mathbb{R})\) and \(s>0\),
	\[
	(f*D_{s}g)(x)=\bigl((D_{1/s}f)*g\bigr)(x/s).
	\]
\end{lemma}
\begin{proof}
	Substituting \(t=su\) gives
	\begin{align*}
		(f*D_{s}g)(x) &=s^{-1/2}\int_{\mathbb{R}}f(t)g\!\left(\frac{x-t}{s}\right)dt\\
		&=s^{1/2}\int_{\mathbb{R}}f(su)g(x/s-u)\,du\\
		&=\int_{\mathbb{R}}(D_{1/s}f)(u)g(x/s-u)\,du.
	\end{align*}
\end{proof}
\begin{lemma}\label{jiaohuan}
	For $f\in L^{p}(\mathbb{R}), 1\leq q\leq \infty$, $M_{w}D_{s}f(t)=D_{s}M_{sw}f(t).$
\end{lemma}
\begin{proof}
	According the definitions of the dilation operators and modulation operators, we can easily obtain  
	\begin{eqnarray*}
		M_{w}D_{s}f(t)
		&=&e^{2\pi iw t}D_{s}f(t)
		=s^{-\frac{1}{2}}e^{2\pi iw t}f\left(\frac{t}{s}\right)\\
		&=&s^{-\frac{1}{2}}e^{2\pi isw \frac{t}{s}}f\left(\frac{t}{s}\right)
		=s^{-\frac{1}{2}}M_{sw}f\left(\frac{t}{s}\right)\\
		&=&D_{s}M_{sw}f(t).
	\end{eqnarray*}
\end{proof}
\begin{proposition}\label{16}\cite[Theorem 16]{ref1}
	Let $B>0$ and $\psi\in L^{1}(\mathbb{R})$ be such that
	$$\lim_{\xi\rightarrow0}\xi^{-l}\mathcal{F}(\psi)(\xi)=c\in\mathbb{C}\setminus\{0\},$$
	for some $l\in \mathbb{N}$. Furthermore, let $(a_{k})_{k\in\mathbb{N}}$ be a sequence in $\mathbb{R}_{+}$ such that $a_{k}\rightarrow0$ as $k\rightarrow\infty$. Then, the
	following are equivalent for $f,g\in PW_{B}^{2}$ real-valued on the real line:\\
	i)$\left|W_{\psi}f\left(\frac{m}{4B},a_{k}\right)\right|=\left|W_{\psi}g\left(\frac{m}{4B},a_{k}\right)\right|,\quad m\in\mathbb{Z},k\in\mathbb{N};$\\
	ii)$f=\pm g.$
\end{proposition}

\begin{proposition}\label{18}\cite[Theorem 18]{ref1}
	Let $a>1,b>0$ and $\psi\in L^{1}(\mathbb{R})$ be a wavelet such that
	$$\lim_{\xi\rightarrow0}\xi^{-l}\mathcal{F}(\psi)(\xi)=c\in\mathbb{C}\setminus\{0\},$$
	for some $l\in \mathbb{N}$. Then, any real-valued bandlimited function $f\in L^{2}(\mathbb{R})$ is uniquely determined up to global sign by the phaseless wavelet coefficients
	$$\{|\langle f,D_{a^{-n}}T_{bm}\psi\rangle|:m\in\mathbb{Z},n\in\mathbb{N}\},$$
	where $\langle f,D_{a^{-n}}T_{bm}\psi\rangle$ is called a wavelet system with generator $\psi$ and parameters $a,b$.
\end{proposition}
\section{Gaussian Convolution Measurements}

For \(s>0\), applying translation and dilation to the Gaussian function, we get 
\begin{equation}\label{eq:sigma}
	\sigma_{s} f(b):=\mathcal{W}_{\phi} f(b,s)=(f*D_{s}\phi)(b),
\end{equation}
which is called the Gaussian analyzing transform. The last equality uses that \(\phi\) is real and even.

The next theorem is therefore stated at one arbitrary scale $a$. The proof of the theorem records exactly how this scale $a$ affects the signal bandwidth and the sampling lattice. This bookkeeping is essential because the dilation \(D_{1/a}\) does not preserve the centered bandlimit \(B\).

\begin{theorem}\label{thm:scaled-gaussian}
	Let \(a,B>0\), let \(\xi_{0},\xi_{1}\in\mathbb{R}\) with \(\xi_{0}\neq\xi_{1}\), and let \(f,g\in PW_{B}^{2}\).
	If, for \(i=0,1\),
	\begin{equation}\label{eq:scaled-data}
		\bigl|f*D_{a}(M_{\xi_{i}}\phi)\bigr|(b)=\bigl|g*D_{a}(M_{\xi_{i}}\phi)\bigr|(b), \qquad b\in\frac{1}{4B}\mathbb{Z},
	\end{equation}
	then \(f\sim g\).
\end{theorem}
\begin{proof}
	Set
	\[
	F=D_{1/a}f,\qquad G=D_{1/a}g.
	\]
	We first identify their Fourier supports.
	By formula \eqref{eq:operator-fourier}, we can get
	\begin{equation*}\label{eq:section-three-bandwidth}
		\widehat{F}(\eta)=a^{-1/2}\widehat{f}(\eta/a), \qquad \widehat{G}(\eta)=a^{-1/2}\widehat{g}(\eta/a).
	\end{equation*}
	If \(|\eta|>aB\), then \(|\eta/a|>B\). Since \(f,g\in PW_{B}^{2}\), we have
	\[
	F,G\in PW_{aB}^{2}.
	\]
For each \(i\in\{0,1\}\), by Lemma \ref{lem:modulation-convolution} and Lemma \ref{lem:dilation-convolution}, it follows that
	\begin{align}
		\bigl(f*D_{a}(M_{\xi_{i}}\phi)\bigr)(b) &=\bigl(F*M_{\xi_{i}}\phi\bigr)(b/a)\nonumber\\
		&=e^{2\pi i\xi_{i} b/a} \bigl((M_{-\xi_{i}}F)*\phi\bigr)(b/a).\label{eq:section-three-reduction}
	\end{align}
	The same identity holds with \(g\) and \(G\).
	
	Putting \(x=b/a\), since  $b\in\frac{1}{4B}\mathbb{Z},$ we can get 
	\begin{equation*}\label{eq:section-three-lattice}
		x\in\frac{1}{4aB}\mathbb{Z}.
	\end{equation*}
	Consequently, \eqref{eq:scaled-data} and \eqref{eq:section-three-reduction} imply
	\[
	\bigl|\phi*(M_{-\xi_{i}}F)\bigr|(x)=\bigl|\phi*(M_{-\xi_{i}}G)\bigr|(x), \qquad x\in\frac{1}{4aB}\mathbb{Z},\quad i=0,1.
	\]
	Here convolution is commutative.
	For $F,G\in PW_{aB}^{2}$, applying Proposition \ref{prop:gaussian-uniqueness} with \(w_0 = \xi_0\) and \(w_1 = \xi_1\) yields \(F\sim G\).
	Finally, applying the inverse dilation \(D_{a}\) gives \(f\sim g\).
\end{proof}


\section{Augmented Morlet--Gaussian Measurements}


We first present another representation of the Morlet wavelet transform.
\begin{lemma}\label{lem:morlet-representation}
	Let \(a>0\) and let \(f\in L^{2}(\mathbb{R})\) satisfy
	\begin{equation}\label{eq:H-domain}
		e^{2\pi at_{0}\xi}\widehat{f}\in L^{2}(\mathbb{R}).
	\end{equation}
	Define \(H_{f}^{a}\in L^{2}(\mathbb{R})\) by
	\begin{equation}\label{eq:H-definition}
		\widehat{H_{f}^{a}}(\xi)=e^{2\pi at_{0}\xi}\widehat{f}(\xi) \quad\text{for almost every }\xi\in\mathbb{R},
	\end{equation}
	and set \(s_{a}=\sqrt{2\pi}\,a\).
	Then
	\begin{equation}\label{eq:morlet-representation}
		\mathcal{W}_{\psi} f(b,a)=e^{-t_{0}^{2}/2} \left[\bigl(D_{1/s_{a}}(H_{f}^{a}-f)\bigr)*\phi\right](b/s_{a}).
	\end{equation}
	In particular, the hypotheses hold for every \(f\in PW_{B}^{2}\), in that case \(H_{f}^{a}-f\in PW_{B}^{2}\).
\end{lemma}
\begin{proof}
	By \eqref{eq:H-domain}, the function in \eqref{eq:H-definition} belongs to \(L^{2}\), so \(H_{f}^{a}\) is well defined by Plancherel's theorem.
	Let
	\[
	\psi_{b,a}=T_{b}D_{a}\psi, \qquad \psi_{b,a}(x)=a^{-1/2}\psi\!\left(\frac{x-b}{a}\right).
	\]
	Equation~\eqref{eq:operator-fourier} gives
	\[
	\widehat{\psi_{b,a}}(\xi)=a^{1/2}e^{-2\pi ib\xi}\widehat{\psi}(a\xi).
	\]
	By applying Plancherel’s theorem to \eqref{eq:wavelet-definition}, and then substituting \eqref{eq:morlet-fourier} into \eqref{eq:wavelet-definition}, we obtain
	\begin{align}\label{eq:conjugated-wavelet-fourier}
		\mathcal{W}_{\psi} f(b,a)\nonumber&=\langle f,\psi_{b,a}\rangle=\langle \hat{f},\hat{\psi}_{b,a}\rangle=a^{1/2}\int_{\mathbb{R}}\widehat{f}(\xi) \overline{\widehat{\psi}(a\xi)}e^{2\pi ib\xi}\,d\xi\\\nonumber
		&=a^{\frac{1}{2}}\int_{\mathbb{R}}\widehat{f}(\xi)\sqrt{2}\pi^{\frac{1}{4}}\left( e^{-\frac{(2\pi a \xi-t_{0})^{2}}{2}}-e^{-2\pi^{2}(a\xi)^{2}}e^{-\frac{t_{0}^{2}}{2}}\right)e^{2\pi ib\xi}d\xi\\\nonumber
		&=\sqrt{2}\pi^{\frac{1}{4}}a^{\frac{1}{2}}e^{-\frac{t_{0}^{2}}{2}} \int_{\mathbb{R}}\widehat{f}(\xi)\left(e^{-2\pi^{2}a^{2}\xi^{2}}e^{2\pi a\xi t_{0}}-e^{-2\pi^{2}a^{2}\xi^{2}}\right)e^{2\pi ib\xi}d\xi\\\nonumber
		&=\sqrt{2}\pi^{\frac{1}{4}}a^{\frac{1}{2}}e^{-\frac{t_{0}^{2}}{2}}\int_{\mathbb{R}}(\widehat{f}(\xi)e^{2\pi a\xi t_{0}}-\widehat{f}(\xi))e^{2\pi ib\xi}e^{-2\pi^{2}a^{2}\xi^{2}}d\xi\\
		&=2^{\frac{1}{4}}s_{a}^{1/2}e^{-\frac{t_{0}^{2}}{2}}\int_{\mathbb{R}}(\widehat{H_{f}^{a}}(\xi)-\widehat{f}(\xi))e^{2\pi ib\xi}e^{-2\pi^{2}a^{2}\xi^{2}}d\xi.
	\end{align}
	The products in this proof are integrable by the Cauchy-Schwarz inequality. Obviously, we can get
	\begin{equation*}\label{eq:gaussian-scaling}
		e^{-2\pi^{2}a^{2}\xi^{2}}=e^{-\pi(\sqrt{2\pi}\,a\xi)^{2}}=e^{-\pi(s_{a}\xi)^{2}}=2^{-1/4}\phi(s_{a}\xi).
	\end{equation*}
	Thus \eqref{eq:conjugated-wavelet-fourier} becomes
	\begin{equation}\label{eq:morlet-h-integral}
		\mathcal{W}_{\psi} f(b,a)=e^{-t_{0}^{2}/2}s_{a}^{1/2} \int_{\mathbb{R}}(\widehat{H_{f}^{a}}(\xi)-\widehat{f}(\xi))\phi(s_{a}\xi) e^{2\pi ib\xi}\,d\xi.
	\end{equation}
	
	It remains to identify this integral as a convolution. By the convolution theorem and \eqref{eq:operator-fourier}, we can obtain 
	\begin{align*}
		\left[\bigl(D_{1/s_{a}}(H_{f}^{a}-f)\bigr)*\phi\right](b/s_{a})
		&=\int_{\mathbb{R}}\left[\bigl(D_{1/s_{a}}(H_{f}^{a}-f)\bigr)*\phi\right]^\wedge(\eta)e^{2\pi i(b/s_{a})\eta}\,d\eta\\
		&=\int_{\mathbb{R}}\bigl(D_{1/s_{a}}(H_{f}^{a}-f)\bigr)^\wedge(\eta)\phi(\eta) e^{2\pi i(b/s_{a})\eta}\,d\eta\\
		&=\int_{\mathbb{R}}s_{a}^{-1/2}\bigl(H_{f}^{a}-f\bigr)^\wedge(\eta/s_{a}) \phi(\eta)e^{2\pi i(b/s_{a})\eta}\,d\eta\\
		&\overset{\eta=s_{a}\xi}{=}s_{a}^{1/2}\int_{\mathbb{R}}(\widehat{H_{f}^{a}}(\xi)-\widehat{f}(\xi))\phi(s_{a}\xi) e^{2\pi ib\xi}\,d\xi.
	\end{align*}
	Comparison with \eqref{eq:morlet-h-integral} proves \eqref{eq:morlet-representation}.
	
	Finally, suppose \(f\in PW_{B}^{2}\).
	The multiplier \(e^{2\pi at_{0}\xi}\) is bounded on the compact interval \([-B,B]\).
	Hence \(e^{2\pi at_{0}\xi}\widehat{f}\in L^{2}\), so \eqref{eq:H-domain} holds.
	Moreover,
	$$(H_{f}^{a}-f)^\wedge(\xi)=\bigl(e^{2\pi at_{0}\xi}-1\bigr)\widehat{f}(\xi)$$
	and therefore \(H_{f}^{a}-f\in PW_{B}^{2}\).
\end{proof}

For \(a>0\) and \(w\in\mathbb{R}\), let
\begin{equation}\label{eq:augmented-parameters}
	s_{a}=\sqrt{2\pi}\,a,\qquad \delta_{a,w}=\frac{w}{s_{a}},\qquad c_{a,w}=e^{2\pi at_{0}\delta_{a,w}}.
\end{equation}
Define the augmented measurement operator
\begin{align}
	\mathcal{A}_{a,w}f(b) &:={\mathcal{W}}_{\psi}\bigl(c_{a,w}M_{-\delta_{a,w}}f\bigr)(b,a)\nonumber\\
	&\quad+e^{-t_{0}^{2}/2}\, \sigma_{s_{a}}\!\left(M_{-\delta_{a,w}} \bigl((c_{a,w}-1)f\bigr)\right)(b).\label{eq:augmented-operator}
\end{align}
This is an augmented Morlet--Gaussian measurement: its first summand is a Morlet transform of a known modulation and scalar weighting of \(f\), while its second summand is a Gaussian convolution term.\\
Define the fixed companion wavelet
\begin{equation}\label{eq:companion-window}
	h_{a,w} :=c_{a,w}M_{\delta_{a,w}}D_{a}\psi +e^{-t_{0}^{2}/2}(c_{a,w}-1) M_{\delta_{a,w}}D_{s_{a}}\phi.
\end{equation}
Straightforward calculation shows 
$$D_{a}\psi(x)=\bigl(e^{it_{0}x/a}-e^{-t_{0}^{2}/2}\bigr)D_{s_{a}}\phi(x).$$
Substitution into \eqref{eq:companion-window} yields
\[
h_{a,w}(x)=e^{2\pi i\delta_{a,w}x} \bigl(c_{a,w}e^{it_{0}x/a}-e^{-t_{0}^{2}/2}\bigr)D_{s_{a}}\phi(x).
\]

\begin{lemma}[Fixed companion-wavelet form]\label{lem:companion-window}
	For every \(a>0\) and \(w\in\mathbb{R}\), the function in \eqref{eq:companion-window} is a nonzero admissible wavelet. More precisely,
	\[
	h_{a,w}\in\mathcal{S}(\mathbb{R})\subset L^2(\mathbb{R}),\qquad \widehat{h_{a,w}}(0)=0,\qquad
	\int_{\mathbb{R}}\frac{|\widehat{h_{a,w}}(\xi)|^{2}}{|\xi|}\,d\xi<\infty.
	\]
	Moreover, for every \(f\in L^{2}(\mathbb{R})\) and \(b\in\mathbb{R}\),
	\begin{equation}\label{eq:companion-window-identity}
		\mathcal{A}_{a,w}f(b)=e^{-2\pi i\delta_{a,w}b}\langle f,T_{b}h_{a,w}\rangle=e^{-2\pi i\delta_{a,w}b}\mathcal{W}_{h_{a,w}} f(b,1).
	\end{equation}
	In particular,
	\begin{equation}\label{eq:companion-window-modulus}
		|\mathcal{A}_{a,w}f(b)|=|\langle f,T_{b}h_{a,w}\rangle|=|\mathcal{W}_{h_{a,w}} f(b,1)|.
	\end{equation}
\end{lemma}
\begin{proof}
	It is easily known that, for every $q\in L^{2}(\mathbb{R}),a>0$,
	$$(q*D_{a}\phi)(b)=\langle q,T_{b}D_{a}\phi\rangle.$$
	According to \eqref{eq:augmented-operator}, by \(M_{\delta_{a,w}} T_{b}=e^{2\pi i\delta_{a,w} b}T_{b}M_{\delta_{a,w}}\), we can get
	\begin{align*}
		\mathcal{A}_{a,w}f(b)
		&={\mathcal{W}}_{\psi}\bigl(c_{a,w}M_{-\delta_{a,w}}f\bigr)(b,a)+e^{-t_{0}^{2}/2}\sigma_{s_{a}}\!\left(M_{-\delta_{a,w}} \bigl((c_{a,w}-1)f\bigr)\right)(b)\\
		&=\bigl((c_{a,w}M_{-\delta_{a,w}}f)\ast D_{a}\psi\bigr)(b)+e^{-t_{0}^{2}/2}\left[ \bigl(M_{-\delta_{a,w}}(c_{a,w}-1)f\bigr)\ast D_{s_{a}}\phi\right](b)\\
		&=c_{a,w}\langle M_{-\delta_{a,w}}f,T_{b} D_{a}\psi\rangle+e^{-t_{0}^{2}/2}(c_{a,w}-1)\langle M_{-\delta_{a,w}}f,T_{b} D_{s_{a}}\phi\rangle\\
		&=c_{a,w}\langle f,M_{\delta_{a,w}}T_{b} D_{a}\psi\rangle
		+e^{-t_{0}^{2}/2}(c_{a,w}-1)\langle f,M_{\delta_{a,w}}T_{b} D_{s_{a}}\phi\rangle\\
		&=e^{-2\pi i\delta_{a,w} b}\bigl(c_{a,w}\langle f,T_{b}M_{\delta_{a,w}} D_{a}\psi\rangle+e^{-t_{0}^{2}/2}(c_{a,w}-1)\langle f,T_{b}M_{\delta_{a,w}} D_{s_{a}}\phi\rangle\bigr)\\
		&=e^{-2\pi i\delta_{a,w} b}\bigl(\langle f,T_{b}\bigl(c_{a,w}M_{\delta_{a,w}} D_{a}\psi+e^{-t_{0}^{2}/2}(c_{a,w}-1) M_{\delta_{a,w}}D_{s_{a}}\phi\bigr)\rangle\bigr)\\
		&=e^{-2\pi i\delta_{a,w} b}\langle f,T_{b}h_{a,w}\rangle\\
		&=e^{-2\pi i\delta_{a,w}b}\mathcal{W}_{h_{a,w}} f(b,1),
	\end{align*}
	which proves \eqref{eq:companion-window-identity} and \eqref{eq:companion-window-modulus}.
	
	It remains to verify the wavelet properties. Since $\psi,\phi\in \mathcal{S}(\mathbb{R})$,
	the explicit formula \eqref{eq:companion-window} also gives \(h_{a,w}\in\mathcal{S}(\mathbb{R})\).
	
	According to \eqref{eq:operator-fourier} and \eqref{eq:companion-window}, we know that
	\begin{align*}
		\hat{h}_{a,w}(\xi)&=c_{a,w}(M_{\delta_{a,w}}D_{a}\psi)^\wedge(\xi)+e^{-t_{0}^{2}/2}(c_{a,w}-1) (M_{\delta_{a,w}}D_{s_{a}}\phi)^\wedge(\xi)\\
		&=c_{a,w}(D_{a}\psi)^\wedge(\xi-\delta_{a,w})+e^{-t_{0}^{2}/2}(c_{a,w}-1) (D_{s_{a}}\phi)^\wedge(\xi-\delta_{a,w})\\		&=c_{a,w}a^{1/2}\hat{\psi}(a\xi-a\delta_{a,w})+e^{-t_{0}^{2}/2}(c_{a,w}-1)s_a^{1/2} \hat{\phi}(s_a\xi-s_a\delta_{a,w}).
	\end{align*}
	Since \(\hat{\psi}, \hat{\phi} \in \mathcal{S}(\mathbb{R})\), \(\widehat{h_{a,w}}\in\mathcal{S}(\mathbb{R})\). When $\xi=0$, the above formulation can be written as
	\begin{align*}
		\hat{h}_{a,w}(0)
		&= c_{a,w}a^{1/2}\hat{\psi}(-a\delta_{a,w})
		+ e^{-t_{0}^{2}/2}(c_{a,w}-1)s_a^{1/2} \hat{\phi}(-s_a\delta_{a,w})\\
		&= c_{a,w}a^{1/2}\sqrt{2}\,\pi^{1/4}e^{-t_{0}^{2}/2}
		e^{-2\pi^{2}(-a\delta_{a,w})^{2}}
		\bigl(e^{2\pi t_{0}(-a\delta_{a,w})}-1\bigr)\\
		&\quad + e^{-t_{0}^{2}/2}(c_{a,w}-1)s_a^{1/2}2^{1/4}
		e^{-\pi (-s_a\delta_{a,w})^{2}}\\
		&=a^{1/2}\sqrt{2}\,\pi^{1/4}e^{-t_{0}^{2}/2}[e^{\sqrt{2\pi}t_{0}w}
		e^{-\pi w^{2}}
		\bigl(e^{-\sqrt{2\pi} t_{0}w}-1\bigr)+(e^{\sqrt{2\pi}t_{0}w}-1)e^{-\pi w^{2}}]\\
		&=a^{1/2}\sqrt{2}\,\pi^{1/4}e^{-t_{0}^{2}/2}[e^{-\pi w^{2}}-e^{\sqrt{2\pi} t_{0}w}e^{-\pi w^{2}}+e^{\sqrt{2\pi} t_{0}w}e^{-\pi w^{2}}-e^{-\pi w^{2}}]\\
		&=0.
	\end{align*}
	Finally, we prove $\int_{\mathbb{R}}\frac{|\widehat{h_{a,w}}(\xi)|^{2}}{|\xi|}\,d\xi<\infty$. It is easily to see that
	\begin{align*}
		\int_{\mathbb{R}}\frac{|\widehat{h_{a,w}}(\xi)|^{2}}{|\xi|}\,d\xi=\int_{|\xi|\leq1}\frac{|\widehat{h_{a,w}}(\xi)|^{2}}{|\xi|}\,d\xi+\int_{|\xi|>1}\frac{|\widehat{h_{a,w}}(\xi)|^{2}}{|\xi|}\,d\xi.
	\end{align*}
	On the one hand, for \(|\xi|\leq1\), since $\widehat{h_{a,w}}\in \mathcal{S}(\mathbb{R})$, Lagrange Mean Value Theorem on the interval between 0 and $\xi$ gives
	$$|\widehat{h_{a,w}}(\xi)|\leq|\xi|\ C_0,$$
	where $C_0:=\mathrm{sup}_{|u|\leq1}|(\widehat{h_{a,w}})'(u)|$. Therefore, we know that
	$$\int_{|\xi|\leq1}\frac{|\widehat{h_{a,w}}(\xi)|^{2}}{|\xi|}\,d\xi\leq\int_{|\xi|\leq1}\frac{|\xi|\ ^2 C_0^2}{|\xi|}\,d\xi=\int_{|\xi|\leq1}|\xi|\ C_0^2\,d\xi\leq C_0^2\int_{|\xi|\leq1}|\xi|\ d\xi<\infty.$$
	On the other hand, for \(|\xi|>1\), one has \(|\xi|^{-1}\leq1\).
	Therefore, we can get 
	$$\int_{|\xi|>1}\frac{|\widehat{h_{a,w}}(\xi)|^{2}}{|\xi|}\,d\xi\leq\int_{|\xi|>1}|\widehat{h_{a,w}}(\xi)|^{2}\,d\xi\leq\int_{\mathbb{R}}|\widehat{h_{a,w}}(\xi)|^{2}\,d\xi<\infty.$$
	These two estimates prove the stated admissibility integral.
\end{proof}

\begin{lemma}[The augmented Gaussian identity]\label{lem:augmented-identity}
	Let \(B>0\), \(a>0\), \(w\in\mathbb{R}\), and \(f\in PW_{B}^{2}\).
	With the notation in \eqref{eq:augmented-parameters},
	\begin{equation*}
		\mathcal{A}_{a,w}f(b)=e^{-t_{0}^{2}/2} \left[\left(M_{-w}D_{1/s_{a}}(H_{f}^{a}-f)\right)*\phi\right](b/s_{a})
	\end{equation*}
	for every \(b\in\mathbb{R}\).
\end{lemma}
\begin{proof}
	The modulation formula in \eqref{eq:operator-fourier} gives
	\begin{align}
		(c_{a,w}M_{-\delta_{a,w}}f)^\wedge(\xi) &= c_{a,w}\widehat{f}(\xi+\delta_{a,w}), \label{eq:v-q-fourier}\\
		(M_{-\delta_{a,w}}\bigl((c_{a,w}-1)f\bigr))^\wedge(\xi) &= (c_{a,w}-1)\widehat{f}(\xi+\delta_{a,w}). \label{eq:v-q-fourier1}
	\end{align}
	
	We now compute the action of \(H^{a}\) on the modulated signal. According to (\ref{eq:H-definition}), we know that
	\begin{align}
		(H_{c_{a,w}M_{-\delta_{a,w}}f}^{a})^\wedge(\xi) &=e^{2\pi at_{0}\xi}[c_{a,w}M_{-\delta_{a,w}}f]^\wedge(\xi)\nonumber\\
		&=e^{2\pi at_{0}\xi}c_{a,w}\widehat{f}(\xi+\delta_{a,w})\nonumber\\
		&=e^{2\pi at_{0}(\xi+\delta_{a,w})}\widehat{f}(\xi+\delta_{a,w})\nonumber\\
		&=(M_{-\delta_{a,w}}H_{f}^{a})^\wedge(\xi).\label{eq:H-modulated}
	\end{align}
	Therefore, subtracting \eqref{eq:v-q-fourier} from \eqref{eq:H-modulated} gives    
	\begin{eqnarray}
		&&\nonumber(H_{c_{a,w}M_{-\delta_{a,w}}f}^{a}-c_{a,w}M_{-\delta_{a,w}}f)^\wedge(\xi)\\
		&=&e^{2\pi at_{0}\xi}c_{a,w}\widehat{f}(\xi+\delta_{a,w})-c_{a,w}\widehat{f}(\xi+\delta_{a,w})\nonumber\\
		&=&\bigl(e^{2\pi at_{0}\xi}-1\bigr)c_{a,w}\widehat{f}(\xi+\delta_{a,w})\nonumber\\
		&=&\bigl(e^{2\pi at_{0}(\xi+\delta_{a,w})}-1\bigr)\widehat{f}(\xi+\delta_{a,w}) -(c_{a,w}-1)\widehat{f}(\xi+\delta_{a,w})\nonumber\\
		&=&[M_{-\delta_{a,w}}(H_{f}^{a}-f)]^\wedge(\xi)-(c_{a,w}-1)\widehat{f}(\xi+\delta_{a,w}).\nonumber\label{eq:augmented-fourier-calculation}
	\end{eqnarray}
	By Plancherel's theorem and (\ref*{eq:v-q-fourier1}), we can get
	\begin{equation}\label{eq:augmented-time-decomposition}
		H_{c_{a,w}M_{-\delta_{a,w}}f}^{a}-c_{a,w}M_{-\delta_{a,w}}f=M_{-\delta_{a,w}}(H_{f}^{a}-f)-M_{-\delta_{a,w}}\bigl((c_{a,w}-1)f\bigr).
	\end{equation}
	Using \eqref{eq:morlet-representation} and \eqref{eq:augmented-time-decomposition}, the Morlet summand in \eqref{eq:augmented-operator} is
	\begin{eqnarray}
		&&\nonumber\mathcal{W}_{\psi}(c_{a,w}M_{-\delta_{a,w}}f)(b,a)\\ 
		&=&e^{-t_{0}^{2}/2} \left[\bigl(D_{1/s_{a}}(H_{c_{a,w}M_{-\delta_{a,w}}f}^{a}-c_{a,w}M_{-\delta_{a,w}}f)\bigr)*\phi\right](b/s_{a})\nonumber\\
		&=&e^{-t_{0}^{2}/2} \left[\bigl(D_{1/s_a}\bigl(M_{-\delta_{a,w}}(H_{f}^{a}-f)-M_{-\delta_{a,w}}\bigl((c_{a,w}-1)f\bigr)\bigr)\bigr)*\phi\right](b/s_a)\nonumber\\
		&=&e^{-t_{0}^{2}/2} \left[\bigl(D_{1/s_a}(M_{-\delta_{a,w}}(H_{f}^{a}-f)\bigr)*\phi\right](b/s_a)\nonumber\\
		&&-e^{-t_{0}^{2}/2} \left[\bigl(D_{1/s_a}(M_{-\delta_{a,w}}\bigl((c_{a,w}-1)f\bigr))\bigr)*\phi\right](b/s_a).\label{eq:morlet-split-before-correction}
	\end{eqnarray}
	On the other hand, the definition of \(\sigma_{s}\) and Lemma \ref{lem:dilation-convolution} give
	\begin{align}\label{eq:sigma-correction}
		\sigma_{s_a}(M_{-\delta_{a,w}}\bigl((c_{a,w}-1)f\bigr))(b)
		&=[(M_{-\delta_{a,w}}\bigl((c_{a,w}-1)f\bigr))*D_{s_a}\phi](b)\nonumber\\
		&=\left[\bigl(D_{1/s_a}M_{-\delta_{a,w}}\bigl((c_{a,w}-1)f\bigr)\bigr)*\phi\right](b/s_a).
	\end{align}
	Consequently, according to the definition of $\mathcal{A}_{a,w}f(b)$ and Lemma \ref{jiaohuan}, by combining (\ref{eq:morlet-split-before-correction}) and (\ref{eq:sigma-correction}), we can get
	\begin{align*}
		\mathcal{A}_{a,w}f(b)
		&= \mathcal{W}_{\psi}\bigl(c_{a,w}M_{-\delta_{a,w}}f\bigr)(b,a)
		+ e^{-t_{0}^{2}/2}\sigma_{s_{a}}\!\left(M_{-\delta_{a,w}}
		\bigl((c_{a,w}-1)f\bigr)\right)(b)\\
		&= e^{-t_{0}^{2}/2}
		\left[\bigl(D_{1/s_a}(M_{-\delta_{a,w}}(H_{f}^{a}-f))\bigr)*\phi\right](b/s_a)\\
		&\quad - e^{-t_{0}^{2}/2}
		\left[\bigl(D_{1/s_a}(M_{-\delta_{a,w}}\bigl((c_{a,w}-1)f\bigr))\bigr)*\phi\right](b/s_a)\\
		&\quad + e^{-t_{0}^{2}/2}
		\left[\bigl(D_{1/s_a}(M_{-\delta_{a,w}}\bigl((c_{a,w}-1)f\bigr))\bigr)*\phi\right](b/s_a)\\
		&= e^{-t_{0}^{2}/2}
		\left[\bigl(D_{1/s_a}(M_{-\delta_{a,w}}(H_{f}^{a}-f))\bigr)*\phi\right](b/s_a)\\
		&= e^{-t_{0}^{2}/2}
		\left[\left(M_{-w}D_{1/s_{a}}(H_{f}^{a}-f)\right)*\phi\right](b/s_{a}).
	\end{align*}
\end{proof}

The following theorem shows that the general bandlimited signals can be reconstructed from the wavelet transform magnitude.
\begin{theorem}\label{thm:main}
	Let \(B>0\),\(a>0\),\(t_{0}\in\mathbb{R}\setminus\{0\}\) and \(w\in\mathbb{R}\setminus\{0\}\). 
	For \(f,g\in PW_{B}^{2}\), the following are equivalent:\\
	(i)\(f\sim g\);\\
	(ii)$$
	|\mathcal{W}_{\psi} f(b,a)|=|\mathcal{W}_{\psi}g(b,a)|,$$
	$$|\mathcal{W}_{h_{a,w}} f(b,1)|=|\mathcal{W}_{h_{a,w}} g(b,1)|,$$
	where \(b\in\frac{1}{4B}\mathbb{Z}\).
\end{theorem}
\begin{proof}
	i)$\Longrightarrow$ ii) Obviously, if $f\sim g$, ii) holds.\\
	ii)$\Longrightarrow$ i) Now we suppose that
	$$|\mathcal{W}_{\psi}f(b,a)|=|\mathcal{W}_{\psi}g(b,a)|.$$ 
	Lemma \ref{lem:morlet-representation} gives
	\begin{align*}
		|\mathcal{W}_{\psi}f(b,a)| &=|e^{-t_{0}^{2}/2}(D_{1/s_{a}}(H_{f}^{a}-f)*\phi)(b/s_{a})|.
	\end{align*}
	Therefore, we have
	\begin{align}
		\left|(D_{1/s_{a}}(H_{f}^{a}-f)*\phi)\left(\frac{b}{s_a}\right)\right|=\left|(D_{1/s_{a}}(H_{g}^{a}-g)*\phi)\left(\frac{b}{s_a}\right)\right|.\label{eq:first-reduction}
	\end{align}
	Next, we consider
	$$|\mathcal{W}_{h_{a,w}} f(b,1)|=|\mathcal{W}_{h_{a,w}} g(b,1)|.$$
	Lemma \ref{lem:augmented-identity} gives
	\begin{align*}
		|\mathcal{W}_{h_{a,w}}f(b,1)| =|e^{-t_{0}^{2}/2}\bigl(M_{-w}D_{1/s_{a}}(H_{f}^{a}-f)*\phi\bigr)(b/s_{a})|.
	\end{align*}
	Therefore, we have
	\begin{align}
		\left|\bigl(M_{-w}D_{1/s_{a}}(H_{f}^{a}-f)*\phi\bigr)\left(\frac{b}{s_a}\right)\right|=\left|\bigl(M_{-w}D_{1/s_{a}}(H_{g}^{a}-g)*\phi\bigr)\left(\frac{b}{s_a}\right)\right|.\label{eq:second-reduction}
	\end{align}
	Meanwhile, Lemma \ref{lem:morlet-representation} gives \(H_{f}^{a}-f,H_{g}^{a}-g \in PW_{B}^{2}\). Therefore, by \eqref{eq:bandwidth-rule}, it is easily know that
	\begin{equation*}
		D_{1/s_{a}}(H_{f}^{a}-f),D_{1/s_{a}}(H_{g}^{a}-g)\in PW_{s_{a}B}^{2}.
	\end{equation*}
	Since $b\in\frac{1}{4B}\mathbb{Z},$ we can get $\frac{b}{s_a}\in\frac{1}{4s_aB}\mathbb{Z}.$ 
	Applying Proposition \ref{prop:gaussian-uniqueness} to $D_{1/s_{a}}(H_{f}^{a}-f),D_{1/s_{a}}(H_{g}^{a}-g)$ 
	via \eqref{eq:first-reduction} and \eqref{eq:second-reduction}, with $w_0=0$ and $w_1=w$, we obtain
	$$D_{1/s_{a}}(H_{f}^{a}-f)=e^{i\alpha}D_{1/s_{a}}(H_{g}^{a}-g)$$
	for some \(\alpha\in\mathbb{R}\). Since dilation is invertible, \(H_{f}^{a}-f=e^{i\alpha}(H_{g}^{a}-g)\).  Taking Fourier transforms gives, almost everywhere,
	\begin{eqnarray}\label{eq:final-multiplier}
		\bigl(e^{2\pi at_{0}\xi}-1\bigr)\widehat{f}(\xi)&=&e^{i\alpha}\bigl(e^{2\pi at_{0}\xi}-1\bigr)\widehat{g}(\xi).
	\end{eqnarray}
	Because \(a>0\) and \(t_{0}\neq0\), the multiplier in \eqref{eq:final-multiplier} vanishes only at \(\xi=0\) and is therefore nonzero almost everywhere.Therefore,
	\begin{eqnarray}
		\widehat{f}(\xi)&=&e^{i\alpha}\widehat{g}(\xi).\nonumber
	\end{eqnarray} Plancherel's theorem then gives \(f=e^{i\alpha}g\) in \(L^{2}(\mathbb{R})\).
\end{proof}

\begin{remark}
	The two data families
	$$
	|\mathcal{W}_{\psi} f(b,a)|=|\mathcal{W}_{\psi}g(b,a)|,$$
	$$|\mathcal{W}_{h_{a,w}} f(b,1)|=|\mathcal{W}_{h_{a,w}} g(b,1)|$$
	in Theorem \ref{thm:main} reduce to two Gaussian convolution measurements of auxiliary bandlimited signals $D_{1/s_{a}}(H_{f}^{a}-f),D_{1/s_{a}}(H_{g}^{a}-g)$, with modulation parameters \(0\) and \(w\). This is exactly the setting of Proposition \ref{prop:gaussian-uniqueness}.
\end{remark}
\begin{remark}	
	To compare our main result with proposition \ref{16} and \ref{18}, we use figures 1,2,3 to illustrate that the signals are uniquely determined (up to a global phase) by fewer sampling points.
	
	As depicted in the figures, Figure \ref{fig:Figure 1} illustrates the sampling points of our main result, while Figure \ref{fig:Figure 3} depicts the sampling points of Proposition \ref{16}, and Figure \ref{fig:Figure 4} displays the sampling points of Proposition \ref{18}. Obviously, we can recover a signal (up to a global phase) using fewer sampling points now. Our sampling frequency has been greatly reduced when compared to \cite{ref1}.
	\begin{figure}[htbp]
		\centering
		\includegraphics[width=0.8\textwidth]{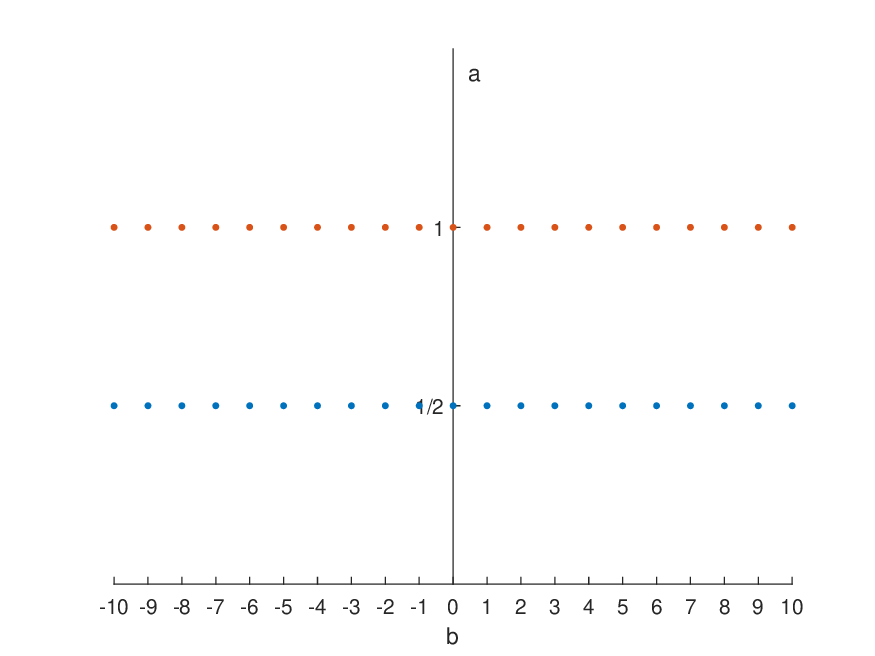}
		\caption{The sampling points of our main result as \(a=\frac{1}{2},B=\frac{1}{4}\).}
		\label{fig:Figure 1}
	\end{figure}
	
	\begin{figure}[htbp]
		\centering
		\includegraphics[width=0.8\textwidth]{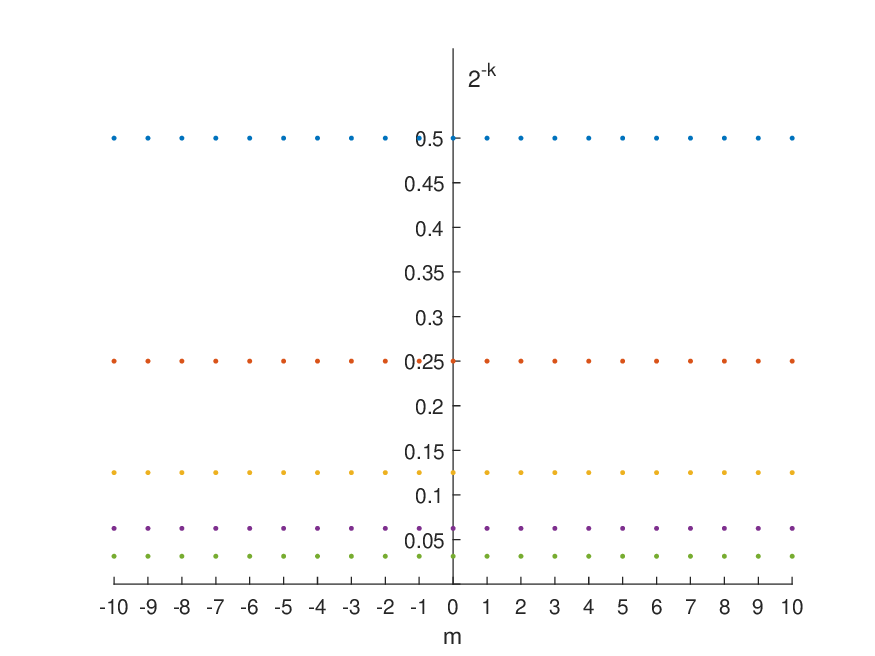}
		\caption{The sampling points of Proposition \ref{16} as \(\{a_{k}\}=2^{-k},B=\frac{1}{4}\), where \(k\in \mathbb{N}\).}
		\label{fig:Figure 3}
	\end{figure}
	
	\begin{figure}[htbp]
		\centering
		\includegraphics[width=0.8\textwidth]{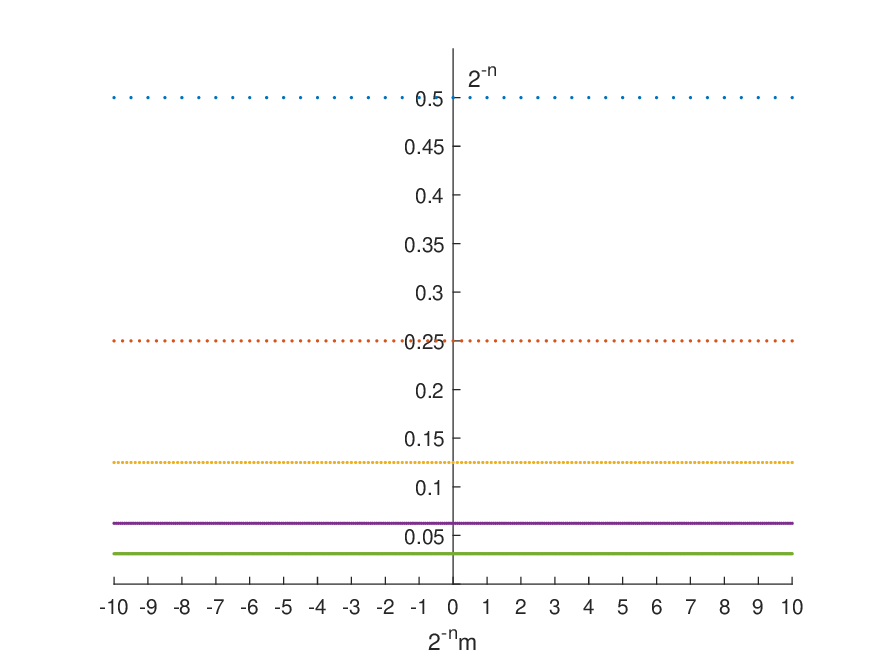}
		\caption{The sampling points of Proposition \ref{18} as \(a=2,b=1\).}
		\label{fig:Figure 4}
	\end{figure}
		\FloatBarrier
\end{remark}

\section{Conclusion}

In this paper, we have established that complex bandlimited signals can be uniquely recovered up to a global phase from two families of coefficient magnitudes, sampled at twice the Nyquist rate in two frequency bins. These two families are associated with translations of the fixed Morlet wavelet and an explicitly constructed fixed companion wavelet, the latter admitting an augmented Morlet--Gaussian representation. The idea of our proof rewrites the two families as two modulated Gaussian convolution measurements of an auxiliary bandlimited signal. By applying a Gaussian convolution uniqueness result, we prove that samples on \((4B)^{-1}\mathbb{Z}\) determine every signal in \(PW_B^2\) up to a global phase. Compared with earlier wavelet phase retrieval results that require wavelets with finitely vanishing moments, our approach achieves uniqueness with fewer sampling points. Future work will address stability and the question of whether Morlet wavelet transform magnitudes alone can achieve phase retrieval.

\section*{Acknowledgments}
The authors would like to thank an anonymous reviewer for their insightful remarks on an earlier draft, which significantly contributed to the improvement of this manuscript.

\section*{Declarations}
\textbf{Generative AI and AI-assisted technologies in the manuscript preparation process}
During the preparation of this work, the authors used ChatGPT for language editing. It was also used to assist in deriving that the companion wavelet in Lemma \ref{lem:companion-window} satisfies the wavelet admissibility condition. After using this tool, the authors reviewed and edited the content as needed.\\
\textbf{Funding} 
This work was supported partially by the National Natural Science Foundation of China (12471131 and 11971348)
and Tianjin Natural Science Foundation Project (24JCYBJC00480).\\
\textbf{Conflict of interest}
The authors declare that they have no Conflict of interest.\\
\textbf{Consent for Publication}
Not applicable.\\
\textbf{Data availability}
No data was used for the research described in the article.

\end{document}